\documentclass[11pt,a4paper,english]{amsart}
\usepackage{babel}
\usepackage[latin1]{inputenc}
\usepackage{amsmath}
\usepackage{amsthm}
\usepackage{amsfonts}
\usepackage{indentfirst}
\usepackage{graphicx}

\usepackage{amssymb}
\usepackage[mathscr]{eucal}

\newtheorem{cor}{Corollary}
\newtheorem{teo}{Theorem}
\newtheorem{rem}{Remark}
\newtheorem{lem}{Lemma}
\newtheorem{exa}{Example}

\newcommand{\co}{{\mathcal O}}

\newcommand{\gp}{\mathbb{P}}

\renewcommand{\int}{{\rm int}}

\newcommand{\betabarra}{\bar{\beta}}
\newcommand{\bb}{\bar{\beta}}
\title[Cone of curves and Cox ring of rational surfaces]{The cone of curves and the Cox ring of rational surfaces given by divisorial valuations}
\author{C.~Galindo \and F.~Monserrat}
\curraddr{\texttt{Carlos Galindo:} Instituto Universitario de Matem\'aticas y Aplicaciones de
Castell\'on and Departamento de Matem\'aticas, Universitat Jaume I,
Campus de Riu Sec. s/n, 12071 Castell\'{o} (Spain).} \email{
galindo@mat.uji.es} \curraddr{\texttt{Francisco Monserrat:} Instituto Universitario de
Matem\'atica Pura y Aplicada, Universidad Polit\'ecnica de Valencia,
Camino de Vera s/n, 46022 Valencia (Spain).}
\email{framonde@mat.upv.es}
\date{}
\thanks{Supported by Spain Ministry of Economy
 MTM2012-36917-C03-03 and Univ. Jaume I P1-1B2012-04}
\date{}

\begin{document}

\maketitle

\begin{abstract}
We consider surfaces $X$ defined by plane divisorial valuations $\nu$ of the quotient field of
the local ring $R$ at a closed point $p$ of the projective plane $\mathbb{P}^2$ over an arbitrary algebraically closed field $k$ and centered at $R$. We prove that the regularity of the cone of curves of $X$ is equivalent to the fact that $\nu$ is non positive on ${\mathcal O}_{\mathbb{P}^2}(\mathbb{P}^2\setminus L)$, where $L$ is a certain line containing $p$. Under these conditions, we characterize when the characteristic cone of $X$  is closed and its Cox ring finitely generated. Equivalent conditions to the fact that $\nu$ is negative on ${\mathcal O}_{\mathbb{P}^2}(\mathbb{P}^2\setminus L) \setminus k$ are also given.
\end{abstract}

\section{Introduction}

In this paper we consider rational surfaces $X$ defined by simple finite sequences of point blowing-ups starting with a blow-up at a closed point of the projective plane over an arbitrary algebraically closed field. We characterize those surfaces $X$ whose cone of curves is regular and determine which ones of these are Mori dream spaces. Our surfaces are intimately related with algebraic objects, plane divisorial valuations, which will be useful in our study.

Ideas by Hensel gave rise to the concept of valuation, established in 1912 by K\"{u}rsch\'ak. This concept is an important tool in several areas of mathematics. One of the fields where valuations are very useful is that of resolution of singularities. Between 1940 and 1960, Zariski and Abhyankar put the foundations of the valuation theory applied to resolution of singularities of algebraic varieties \cite{ja21,ja22,ja23,ja1,ja2}. It is well-known that resolution in characteristic zero was proved by Hironaka without using valuations, however many of the attempts and known results involving resolution in positive characteristic use valuation theory (see \cite{tei} as a sample).

Generally speaking, valuations are not well-known objects although some families of them have been rather studied. This is the case of valuations of quotient fields of regular two-dimensional local rings $(R,m)$ centered at $R$ (plane valuations). Zariski gave a classification for them and a refinement in terms of dual graphs can be found in \cite{ja20} (see also \cite{ja15}). Plane valuations are in one to one correspondence with simple sequences of point blowing-ups starting with the blowing-up of Spec$R$ at $m$. Here, simple means that the center of each blow-up is in the exceptional divisor produced by the previous one. Probably, and from the geometric point of view,  the so-called divisorial valuations are the most interesting ones; in the case of plane valuations, they correspond with finite sequences as above and are defined by the last created exceptional divisor. Within the problem of resolution of singularities, valuations are considered  as a local object and, mostly, used to treat the local uniformization problem. However, as we will see, useful global geometrical properties arise associated with certain classes of plane divisorial valuations.

Along this paper, $k$ will denote an algebraically closed field of arbitrary characteristic, $\gp^2:=\gp^2_k$ the projective plane over $k$ and our valuations will be of the quotient field of the local ring $R:=\mathcal{O}_{\gp^2,p}$, where $p$ is a fixed point in $\gp^2$. To fix notation, we set $(X:Y:Z)$ projective coordinates in $\gp^2$, consider the line $L$ with equation $Z=0$, that will be called the line at infinity, and the point $p$ with projective coordinates $(1:0:0)$. In addition, pick affine coordinates $x=X/Z$; $y=Y/Z$ in the chart of $\gp^2$ given by $Z \neq 0$
and consider a divisorial valuation $\nu$ of the quotient field of the local ring $R$ and centered at $R$. A goal of this paper is either to characterize, or to provide geometrical properties of, the fact that the valuation $\nu$ is non positive or negative on all polynomials $p(x,y)$ in the set $k[x,y] \setminus k$.

As mentioned above, some of our characterizations or properties involve a good behaviour  of interesting global objects as the cone of curves, the characteristic cone or the Cox ring attached to the surface that the sequence of point blowing ups determined by $\nu$ defines.

An interesting class of valuations satisfying one of the facts we  characterize, $\nu(p) \leq 0$ for every $p \in k[x,y]$, is studied in \cite{cpr}. There, the authors deduce good properties concerning curves and line bundles on the surface that certain valuations define. These valuations are those attached with  pencils defined by the line at infinity $L$ and curves with only one place at infinity. The study of these curves was started by Abhyankar and Moh \cite{ja4,ja6} and we recall that a projective curve $C \subset \gp^2$ has only one place at infinity when $C \cap L$ is a unique point $p$ and $C$ is reduced and unibranched at $p$. The mentioned valuations are those defined by the sequence of point blowing-ups eliminating the base points of the mentioned pencil. A close class of valuations, also centered at $R$ but non-necessarily of divisorial type, are those  whose centers (maybe infinitely many) are given by some curve, or family of curves, with only one place at infinity. These valuations have been studied in \cite{london, dcc} because they are useful in coding theory. These valuations satisfy an important property, an Abhyankar-Moh semigroup type theorem, that is a result giving a minimal set of generators for the semigroup $-\nu(k[x,y])$, whenever one considers suitable values of the characteristic of the field $k$ \cite{Galmon}. It is also worth to mention the existence of a family of valuations, satisfying the mentioned property on their signs, which plays an important role for the dynamics of polynomial mappings of the affine plane \cite{fj2,fj3}.

Returning to general divisorial valuations $\nu$, we will prove in Theorem \ref{gordo1} that the fact $\nu(p) \leq 0$ for all $p \in k[x,y]$ is equivalent to the one that the cone of curves $NE(X)$ of the surface $X$ defined by $\nu$ is regular, a stronger property that being rational polyhedral. Two more equivalent conditions, (a) and (c), are also stated in the mentioned theorem. Condition (a) is a very simple to check local property, while Condition (c) is global, the nefness of a certain divisor derived from $\nu$. As a consequence, it is easy to check whether a valuation satisfies the conditions in Theorem \ref{gordo1} and it is clear that these valuations are a natural extension of those in \cite{cpr} and they keep good global properties (see also Remark \ref{llados}).

Recall that cones associated with varieties have been used in the last decades to approach the theory of minimal models. Kawamata's cone theorem \cite{he5}, which generalizes a result by Mori \cite{he9} and implies that the closure of the cone of curves of a variety is rational polyhedral if its anticanonical bundle is ample, is an important ingredient in the model minimal program. For Calabi-Yau varieties, a substitute of the cone theorem is the Morrison-Kawamata conjecture, related with the nef cone, \cite{to29, to30,to19, to}, which is a theorem in dimension 2 \cite{to41,to33, to19}. With respect to the cone of curves, even for the simplest case of surfaces, there is no general result that allows us to decide when it is rational polyhedral. Some other reasons explaining the interest of satisfying this property can be seen in \cite{c-g,ni}. It is worthwhile to add that, for surfaces, apart from Kawamata's result, there exists a number of results guarantying good properties of the cone of curves (see \cite{cpr,ijm,hel,l-m} for instance).

Two equivalent conditions to the fact $\nu(p) < 0$ for all $p \in k[x,y]\setminus k$ are also given in Theorem \ref{gordo2}. They are  of either of global type or a mix of global and local data.

The characteristic cone is other interesting cone which can be attached with a variety. It was introduced by Hironaka, Mumford and Kleiman  and, for instance, its relative version is useful in the study of projective birational morphisms $\pi: Z \rightarrow Y$ of normal algebraic varieties which are an
isomorphism outside $\pi^{-1}(O)$, for a closed point $O \in Y$ \cite{he3}. Our Theorem \ref{gordo3} gives an equivalent statement to the fact that the characteristic cone of the surface attached to a valuation which is non positive on polynomials of $k[x,y]$ (valuations previously characterized in Theorem \ref{gordo1}) is closed. Using \cite{Hu,gm}, we prove in Corollary \ref{COX} that the same statement characterizes those surfaces given by valuations as mentioned whose Cox rings are finitely generated.

Cox rings were introduced by Cox in \cite{cox5} for toric varieties to show  that they behave like a projective space in many ways. This definition was extended  to varieties with free, finitely generated Picard group \cite{Hu} and, roughly speaking,  this ring is the graded one of the section of line bundles on the variety. Finite generation of Cox rings achieves great importance in the minimal model program, since for varieties with this property the mentioned program can be carried out for any divisor. Note that recently has been proved \cite{bchm} the existence of minimal models for complex varieties of log general type and that the Cox ring of a Fano complex variety is finitely generated. With respect to surfaces, the fact that the Cox ring is finitely generated is related to invariant theory and the Hilbert's fourteen problem as one can see in \cite{cox16, tot10, to2}. The recent literature contains a number of papers concerning this issue \cite{gm,tot41,fa,al,ot,hw} and confirms that the classification of rational surfaces (and, of course, of varieties) with finitely generated Cox ring is a difficult problem.

Before explaining how this paper is organized, we notice that our results happen for a field
$k$ of any characteristic. For the field of complex numbers, in item a) of \cite[Theorem 1.4]{mon} the author gives a characterization of non positive valuations in terms of what he calls key forms. We show in Corollary \ref{mondal1} that his characterization is equivalent to that of  item (a) of our Theorem \ref{gordo1}. A similar situation holds for item b) of \cite[Theorem 1.4]{mon} and our characterization in item (a) of Theorem \ref{gordo2} (see Corollary \ref{mondal2}).

Apart from the introduction, this paper has two more sections. Section \ref{lados} is devoted to give the mentioned conditions characterizing when the valuation $\nu$ is non positive or negative on $k[x,y]\setminus k$. Proof of Theorem \ref{gordo1} requires two previous lemmas, where contact maximal values of the valuation are involved. This section also contains the above explained relation with recent results in the case of the complex field. Section \ref{latres} contains Theorem \ref{gordo3} and Corollary \ref{COX} on the characteristic cone and Cox ring of $X$. This last result gives a large class of Mori dream surfaces with arbitrarily large Picard number. These include infinite families of surfaces whose anticanonical Iitaka dimension is equal to $- \infty$ (see Example \ref{ejemplo}).

\section{Characterization of non-positive and negative divisorial valuations}
\label{lados}
\subsection{Plane divisorial valuations}
\label{pdv}
A valuation of a field $F$ is a surjective map $\nu: F^* (:=F \setminus\{0\}) \rightarrow G$, where $G$ is a totally ordered commutative group, such that $\nu(f+g) \geq \min \{\nu(f), \nu(g)\}$ and $\nu(fg) = \nu (f) + \nu (g)$, for $f,g \in F$. If $F$ is the quotient field of a local regular ring $(R,m)$, $\nu$ is said to be centered at $R$ whenever $R \cap m_\nu = m$, where $m_\nu = \{f \in F | \nu(f) >0 \} \cup \{0\}$ is the maximal ideal of the valuation ring $R_\nu = \{f \in F | \nu(f) \geq 0 \} \cup \{0\}$.

Set $\mathbb{P}^2 := \mathbb{P}^2_k$ the projective plane over an arbitrary algebraically closed field $k$ and $p$ a closed point in $\mathbb{P}^2$ as we said in the introduction. In this paper, $R$ will be the local ring ${\mathcal O}_{{\mathbb P}^2,p}$ at the point $p$. Valuations centered at $R$ will be called {\it plane valuations} and they are in 1-to-1 correspondence with simple sequences $\pi$ of point blowing-ups over $\gp^2$ \cite{ja20}. This means that we start by blowing-up $\gp^2$ at $p$ and, afterwards, each blowing-up has center in the last created exceptional divisor. Plane valuations can be classified in five types depending on the relative position of  the obtained exceptional divisors \cite{ja20}. Sequences $\pi$ need not to be finite and finite ones correspond with the so-called (plane) {\it divisorial valuations}. This designation  comes from the fact that they are defined by the last created exceptional divisor in its attached sequence $\pi$. The value group of a divisorial valuation is that of integer numbers and this type of plane valuations constitute the only one satisfying that the transcendence degree of the field $R_\nu / m_\nu$ over $k$ is $1$.

We will consider local coordinates $\{u,v\}$, $u=Y/X$ and $v=Z/X$, around the point $p=p_1=(1:0:0)$ of $\gp^2$. Set $\nu$ a divisorial valuation  of the quotient field of $R$ centered at $R$ and
\begin{equation}\label{blow}
\pi:   X = X_{m} \stackrel{\pi_{m}}{\longrightarrow}
X_{m-1} \longrightarrow \cdots \longrightarrow X_{1}
\stackrel{\pi_{1}}{\longrightarrow} X_0 = \mathbb{P}^2_k
\end{equation}
the simple sequence of point blowing-ups that $\nu$ defines. Here $\pi_1$ is the blowing-up of $\mathbb{P}^2_k$ at $p_1$ and $\pi_{i+1}$, $1 \leq i \leq m-1$, the blowing-up of $X_i$ at the unique point $p_{i+1}$ of the exceptional divisor defined by $\pi_i$, $E_i$, such that $\nu$ is centered at the local ring $\mathcal{O}_{X_i,p_{i+1}}$. Denote by $C_\nu := \{p_i\}_{i=1}^{m}$ the sequence (or configuration) of infinitely near points above defined; $p_i$ is named to be {\it proximate} to $p_j$ ($p_i \rightarrow p_j$) whenever $i>j$ and $p_i$ belongs either to $E_{j}$ or to the strict transform of $E_{j}$ on $X_{i-1}$. A point $p_i$ is {\it satellite} whenever there exists $j < i-1$ satisfying $p_i \rightarrow p_j$; otherwise, it is called {\it free}. The above mentioned relative position among (strict transforms on $X$ of) divisors given by $\pi$ is usually represented by the dual graph of $\nu$. This is  a tree where each vertex represents a divisor and two vertices are joined whenever their corresponding divisors meet. The dual graph is an equivalent datum to the structure of the Hamburger-Noether expansion of the valuation which provides parametric equations for it \cite{d-g-n}. Both of them  determine, and can be determined by, a finite sequence of positive rational numbers $\{\beta'_j\}_{j=0}^{g+1}$, usually called the Puiseux exponents of the valuation \cite[(1.5.2)]{d-g-n}.  In this paper, we will use this sequence and also another family of integer numbers $\{\overline{\beta}_j\}_{j=0}^{g+1}$, the so-called sequence of maximal contact values \cite[(1.5.3)]{d-g-n} (see \cite[Lemma 1.8]{d-g-n}). This last sequence spans the semigroup of values of $\nu$, $S(R):= \nu(R \setminus \{0\})$. In fact $\{\overline{\beta}_j\}_{j=0}^{g}$ is a minimal set of generators of $S(R)$ and $\overline{\beta}_{g+1}$ detects the number of free points after the last satellite one of $\pi$ and it will be an important value for us. We conclude by noticing that Puiseux exponents, contact maximal values and dual graph are data that can be obtained one from each one of the others \cite[Theorem 1.11]{d-g-n}.

\subsection{Cones and Cox ring}

This brief section will be devoted to introduce some objects that will appear in our results. Although they can be defined in general varieties, we will suppose that $X$ is a rational surface defined by the sequence $\pi$ given by a divisorial plane valuation $\nu$ as above. For a start, set Pic$(X)$ the Picard group of the surface $X$ and Pic$_\mathbb{Q} (X) = \mathrm{Pic} (X) \otimes_\mathbb{Z} \mathbb{Q}$ the corresponding vector space over the field of rational numbers. It is well-known that the intersection form extends to a bilinear pairing over Pic$_\mathbb{Q} (X)$ which will be denoted by $\cdot$. Denote by $E_0$ a line in $\gp^2$ that does not pass through the point $p$ and set $\{E_i\}_{i=0}^m$ (respectively,  $\{E_i^*\}_{i=0}^m$) the strict (respectively, total) transforms of the line $E_0$ and the exceptional divisors (also denoted $\{E_i\}_{i=1}^m$) on the surface $X$ through $\pi$. Denote $[E_i]$ (respectively,  $[E_i^*]$), $0 \leq i \leq m$, the class modulo linear (or numerical) equivalence on Pic$_\mathbb{Q} (X)$ of the mentioned divisors. Then $\{[E_i]\}_{i=0}^m$ and $\{[E_i^*]\}_{i=0}^m$ are bases of the vector space Pic$_\mathbb{Q} (X)$ and $E_i = E_i^* - \sum_{p_j \rightarrow p_i} E_j^*$ gives a change of basis in the $\mathbb{Q}$-vector space of divisors with exceptional support.

For us, the {\it cone of curves} of $X$, $NE(X)$, is defined as the convex cone of Pic$_\mathbb{Q} (X)$ generated by the classes in Pic$_\mathbb{Q} (X)$ of effective divisors on $X$. When one changes the sentence ``effective divisors" by ``nef divisors" (respectively, ``divisors defining base point free linear systems") one gets the concept of {\it nef cone} $\mathcal{P}(X)$ (respectively, {\it characteristic cone } $\tilde{\mathcal{P}}(X)$). Recall that $\tilde{\mathcal{P}}(X)$ is included in $\mathcal{P}(X)$, both cones have the same topological interior \cite{Kleiman} and  $\mathcal{P}(X)$ is the dual cone of $NE(X)$. Also notice that these cones are called to be regular if they are generated by some elements forming a basis of the $\mathbb{Z}$-module Pic$(X)$.

Denote $\mathbf{s}=(s_0,s_1, \ldots, s_m) \in \mathbb{Z}^{m+1}$, $D_\mathbf{s} := \sum_{i=0}^m s_i E_i$ and regard the vector spaces
\[
H^0 \left(X, \mathcal{O}_X (D_\mathbf{s})\right) = \{f \in k(X)\setminus \{0\} \;| \; \mathrm{div}_X(f) + D_\mathbf{s} \geq 0\} \cup \left\{0\right\}
\]
as $k$-vector subspaces of the function field $k(X)$ of $X$. Then the {\it Cox ring} of $X$ is defined as the graded $k$-subalgebra of $k(X)$
\[
\mathrm{Cox}(X) := \bigoplus_{\mathbf{s}\in \mathbb{Z}^{m+1} } H^0 \left(X, \mathcal{O}_X (D_\mathbf{s})\right),
\]
where we must notice that different bases of Pic$(X)$ give isomorphic (as $k$-algebra) Cox rings.

We conclude this section by recalling a concept which we will also use. This is that of Iitaka dimension, $\kappa(D)$, of a divisor $D$ on $X$:
$$\kappa(D):=\max\{\dim \varphi_{|nD|}(X)\mid n\in N(D)\},$$
where $N(D)=\{m\in \mathbb{Z}_{>0}\mid H^0(X,{\mathcal O}_X(mD))\not=0\}$, $\mathbb{Z}_{>0}$ denotes the set of positive integers, $\dim$ projective dimension and, for each $n\in N(D)$, $\varphi_{|nD|}(X)$ is the closure of the image of the rational map $\varphi_{|nD|}:X \cdots \rightarrow \mathbb{P}H^0(X,{\mathcal O}_X(nD))$ defined by the complete linear system $|nD|$. By convention, $\kappa(D)= - \infty$ whenever $|nD| = \emptyset$ for all $n >0$.

\subsection{The characterization}

Throughout this section, $\nu$ will be a (plane) divisorial valuation such that its corresponding blowing-up sequence is $\pi$ (see (\ref{blow})). For each index $i\in \{1,2,\ldots,m\}$, let $\varphi_i$ be an analytically irreducible germ of curve at $p$ whose strict transform on $X_{i}$ is transversal to $E_i$ at a non-singular point of the exceptional locus. Consider the family of divisors on $X$:
\begin{equation}
\label{Z}
D_i:=d_iE_0^*-\sum_{j=1}^m \mathrm{mult}_{p_j}(\varphi_i) E_j^*,
\end{equation}
$1 \leq i \leq m$, where $d_i$ denotes the intersection multiplicity at $p$ of the germ of the line $L$ at $p$ and  the germ $\varphi_i$, and $\mathrm{mult}_{p_j}(\varphi_i)$ the multiplicity at $p_j$ of the strict transform of $\varphi_i$ at $p_j$. Set $D_0:=E_0$; by Noether's formula it holds that
 $\{[D_i]\}_{i=0}^m$ is the dual basis, with respect to the above defined bilinear pairing, of the basis of Pic$_\mathbb{Q} (X)$ given by $\{[\tilde{L}]\}\cup \{[E_i]\}_{i=1}^m$,  where  $\tilde{L}$ denotes the strict transform of $L$ on $X$.

Fix an ample divisor $H$ on $X$ and define
$$Q(X):=\{x\in {\rm Pic}_{\mathbb{Q}}(X)\mid x^2\geq 0\; \mbox{ and }\; [H]\cdot x\geq 0\}.$$
Then, the following result happens:

\begin{lem}\label{lema1}
Let $D=dE_0^*-\sum_{i=1}^m r_i E_i^*$ be a divisor on $X$, with $d,r_i\in \mathbb{N}$ for all $i=1, 2, \ldots,m$, $\mathbb{N}$ being the nonnegative integer numbers, and such that $d\not=0$. Then, the following conditions are equivalent:
\begin{itemize}
\item[(a)] $[D]\cdot x\geq 0$ for all $x\in Q(X)$.
\item[(b)] $D^2\geq 0$.
\end{itemize}
\end{lem}
\begin{proof}

Notice that $[D]\cdot x\geq 0$ for all $x\in Q(X)$ if and only if $\sum_{i=1}^m r_i x_i\leq d$ for all $(x_1,x_2,\dots,x_m)\in \mathbb{Q}^m$ such that $\sum_{i=1}^m x_i^2\leq 1$. Consider the map $f: \mathbb{R}^m \rightarrow \mathbb{R}$ given by $f(x_1,x_2,\ldots,x_m):=\sum_{i=1}^m r_i x_i$; applying the Lagrange multipliers method it is straightforward to check that $\sqrt{\sum_{i=1}^m r_i^2}$ is the maximum value that $f$ takes in the set
$$\left\{(x_1,\ldots,x_m) \in \mathbb{R}^m \mid \sum_{i=1}^m x_i^2\leq 1\right\}.$$

As a consequence,  $[D]\cdot x\geq 0$ for all $x\in Q(X)$ if, and only if, $\sqrt{\sum_{i=1}^m r_i^2}\leq d$, that is, if and only if  $D^2\geq 0$.

\end{proof}

Let $\{\betabarra_i\}_{i=0}^{g+1}$ be the sequence of maximal contact values of $\nu$, notice that $g$ is the number of characteristic pairs of a general element of $\nu$. This sequence will be useful in the proof of the following result.

\begin{lem}\label{lema2}
Let $\nu$ be a plane (divisorial) valuation of the quotient field of $R\; (:= \mathcal{O}_{\mathbb{P}^2,p})$ centered at $R $ and $X$ the surface that it defines through $\pi$. Consider the family of divisors on $X$ given in (\ref{Z}). Then $D_0^2=1$, $D_1^2=0$ and, for any index $i\in \{2, 3, \ldots,m\}$ such that $D_i^2\geq 0$, the following properties are satisfied:
\begin{itemize}
\item[(a)] If $p_i$ is a satellite point of the configuration $ \mathcal{C}_\nu$ that  defines $\pi$, it holds $D_i^2>0$.
\item[(b)]  $D_{i-1}^2\geq 0$ and, moreover, if $D_{i-1}^2=0$ then either $i=2$, or $p_i$ is satellite and $p_{i-1}$ is free.
\end{itemize}
\end{lem}

\begin{proof}
The equalities $D_0^2=1$ and $D_1^2=0$ are straightforward and, without loosing generality, we can assume that $i=m\geq 2$. Define the sequence of positive integer numbers $\{e_j\}_{j=0}^g$, where $e_j:= \gcd(\betabarra_0,\betabarra_1,\ldots, \betabarra_j)$ (notice that $e_g=1$).

First, we are going to prove item (a).  Reasoning by contradiction, assume that $D_m^2=0$. Since the point $p_m$ is satellite, $\betabarra_{g+1} = e_{g-1}\betabarra_g$ and one gets
$$0=D_m^2=d_m^2-e_{g-1}\betabarra_g=e_{g-1}\left[\frac{d_m}{e_{g-1}}d_m-\betabarra_g\right].$$
Taking into account that $\gcd(\betabarra_0,\betabarra_1,\ldots, \betabarra_g)=1$, the above equality leads to a contradiction because $e_{g-1}$ divides $d_m$ and $\gcd(e_{g-1}, \betabarra_g)=1$.

To prove item (b), notice that, when $p_m$ is a free point and $m\geq 3$, then  $D_{m-1}^2>0$ by Noether's formula. Therefore we will assume that the point $p_m$ is satellite. Denote by $\hat{\nu}$ the divisorial valuation defined by the divisor $E_{m-1}$.
Let $\{\hat{\betabarra}_j\}_{j=0}^{\hat{g}+1}$ be the sequence of maximal contact values of $\hat{\nu}$ and set $\hat{e}_j:=\gcd(\hat{\betabarra}_0,\ldots, \hat{\betabarra}_j)$ for all $j\in \{0,1,\ldots,\hat{g}\}$ and $e:=\hat{e}_{\hat{g}-1}/e_{\hat{g}-1}$. Notice that $g\geq 1$ (because $p_m$ is a satellite point) and either $g=\hat{g}$ or $\hat{g}=g-1$.

First assume that $g=\hat{g}$ (which means that $p_{m-1}$ is also a satellite point). The following equality we are going to prove will be useful.
\begin{equation}
\label{eee}
| \hat{\bb}_{g} - e \bb_{g}|=\frac{1}{e_{g-1}}.
\end{equation}
Indeed, consider the sequences of Puiseux exponents $\{ \beta'_j \}_{j=0}^{g+1}$  and $\{\hat{\beta}'_j\}_{j=0}^{g+1}$ corresponding with $\nu$ and $\hat{\nu}$, respectively. Notice that these exponents can be obtained from continued fractions determined by the dual graphs (or the structures of the Hamburger-Noether expansions) of the valuations $\nu$ and $\hat{\nu}$. By \cite[Lemma 1.8]{d-g-n}, it holds that
\[
\beta'_g=\frac{\bb_g - N_{g-1} \bb_{g-1}}{e_{g-1}} \;\;\; \mathrm{and} \;\;\; \hat{\beta}'_g=\frac{\hat{\bb}_g - \hat{N}_{g-1} \hat{\bb}_{g-1}}{\hat{e}_{g-1}},
\]
where $N_0:=0$ (respectively, $\hat{N}_0:=0$) and $N_j:=\frac{e_{j-1}}{e_j}$ (respectively, $\hat{N}_j:=\frac{\hat{e}_{j-1}}{\hat{e}_j}$) for $j\in \{1,\ldots,g\}$.
Since $N_{g-1} \bb_{g-1}/ e_{g-1} = \hat{N}_{g-1} \hat{\bb}_{g-1} / \hat{e}_{g-1}$, it happens that
\[
\frac{\hat{\bb}_{g} - e \bb_{g}}{\hat{e}_{g-1}} = \hat{\beta}'_g - \beta'_g.
\]

The values $\hat{e}_{g-1}$ and ${e}_{g-1}$ are reduced denominators for $\hat{\beta}'_g$ and $\beta'_g$ and, moreover, if $\langle a_0;a_1,\ldots, a_r\rangle=\langle a_0;a_1,\ldots, a_r-1,1\rangle$ are the two equivalent expressions of $\beta'_g$ as continued fractions (following the notation of \cite[Section 7.2]{ni-zu}), it holds that $\hat{\beta}'_g=\langle a_0;a_1,\ldots, a_r-1\rangle$ and then
$$|\hat{\beta}'_g - \beta'_g | =\frac{1}{\hat{e}_{g-1} {e}_{g-1}}$$
by \cite[Theorem 7.5]{ni-zu}, which proves (\ref{eee}).

Now,
$$D_m^2=d_m^2-e_{g-1}\betabarra_g\geq 0\;\;\mbox{ and }\;\; D_{m-1}^2=e^2d_m^2-\hat{e}_{g-1}\hat{\betabarra}_g$$
because $\betabarra_{g+1}=e_{g-1}\betabarra_g$, $\hat{\betabarra}_{g+1}=\hat{e}_{g-1}\hat{\betabarra}_g$,  since $\nu$ and $\hat{\nu}$ are defined by satellite points, and $d_{m-1}=e\; d_m$. Then
$$D_{m-1}^2=
e^2 \left[d_m^2-\frac{e_{g-1}\hat{\betabarra}_g}{e}\right]\geq e^2\left[d_m^2-e_{g-1}\betabarra_g-\frac{1}{e}\right]=e^2\; \left[D_m^2-\frac{1}{e}\right]> 0,$$
where the first inequality is a consequence of (\ref{eee}) and the last one holds because $D_m^2$ is the product of $e_{g-1}$ by an integer (that must be positive by item (a)) and $1/e<e_{g-1}$.

Finally, assume that $\hat{g}=g-1$. This case occurs if and only if $P_{m-1}$ is free. Therefore $e_{g-1}=2$, and, using Noether's formula, it is easy to deduce that $d_{m-1}=d_m/2$ and
$$\hat{\betabarra}_{\hat{g}+1}=\frac{\betabarra_{g+1}+2}{4}.$$
Using this equality, we have that
$$D_{m-1}^2=\frac{1}{4}\left[d_m^2-\betabarra_{g+1}-2  \right]=\frac{1}{4} D_m^2-\frac{1}{2}.$$
Taking into account that $D_m^2$ must be greater than or equal to $2$ because $D_{m-1}^2$ is an integer, we conclude $D_{m-1}^2\geq 0$ and the result is proved.

\end{proof}

\begin{rem}\label{remark1}
{\rm
Lemma \ref{lema2} shows that, under the assumption $D_m^2\geq 0$, it holds that $D_i^2\geq 0$ for all $i\in \{0,\ldots,m-1\}$ and moreover the divisors $D_i$ whose self-intersection is equal to zero are among those such that $p_i$ is a free point.

}
\end{rem}

Now we are ready to state and prove our first theorem, which provides equivalent statements to the fact that the cone of curves of the rational surface $X$ defined by a valuation $\nu$ as above is regular. Recall that $\mathbb{P}^2$ is the projective plane over an algebraically closed field $k$, $p=p_1=(1:0:0)$, $R$ the local ring $\mathcal{O}_{\mathbb{P}^2,p}$, $L$ is the line at infinity (defined by the equation $Z=0$) and $\{x,y\}$ affine coordinates in the chart $Z \neq 0$.

\begin{teo}\label{gordo1}
Let $\nu$ be a plane divisorial valuation of the quotient field of $R$ centered at $R$. Set $X$ the surface that it defines via its attached sequence  of point blowing-ups $\pi$ (\ref{blow}). Consider  the divisor  $D_m$  defined in (\ref{Z}) and $\bb_{g+1}$ the last maximal contact value of $\nu$.  Then, the  following conditions are equivalent:
\begin{itemize}
\item[(a)] $d_m^2\geq \betabarra_{g+1}$.
\item[(b)] The cone $NE(X)$ is regular.
\item[(c)] $D_m$ is a nef divisor.
\item[(d)] $\nu(f)\leq 0$ for all $f\in k[x,y]$.

\end{itemize}

\end{teo}

\begin{proof}
Let $S_1$ be the convex cone of Pic$_\mathbb{Q} (X)$ spanned by the classes $[\tilde{L}], [E_1],[E_2],\ldots,[E_m]$ and recall that its dual cone
$S_1^{\vee}$ is spanned by $[D_0],[D_1],\ldots, [D_m]$, where $D_i$ are the divisors defined in (\ref{Z}).
Statement (a) is equivalent to say that the divisor $D_m$ has non-negative self-intersection. Assume that it holds and let us prove Statement (b).  By Lemma \ref{lema2}, $D_i\in Q(X)$ for all $i\in \{0, 1, \ldots,m\}$, $Q(X)$ being the set defined before Lemma \ref{lema1}.
Then, setting $\overline{NE}(X)$ the topological closure of the cone $NE(X)$, it holds, on the one hand, that $S_1^{\vee}\subseteq \overline{NE}(X)$ because $Q(X)\subseteq \overline{NE}(X)$. On the other hand,
any class of an irreducible curve on $X$ different from $[\tilde{L}]$, $[E_1], [E_2],\ldots,[E_m]$ must belong to $S_1^{\vee}$. Thus, one gets
$$NE(X)\subseteq S_1+S_1^{\vee}\subseteq \overline{NE}(X),$$
where $+$ means Minkowski sum. Taking topological closures, the equality $\overline{NE}(X)=S_1+S_1^{\vee}$ is obtained. Lemma \ref{lema1} implies that $Q(X)\subseteq (S_1^{\vee})^{\vee}=S_1$, since $D_i^2\geq 0$ for all $i=0, 1, \ldots,m$. Therefore, $S_1^{\vee}\subset S_1$ because $D_i\in Q(X)$ for all $i$ and, hence, $\overline{NE}(X)=S_1$. Since $S_1$ is finitely generated by effective classes we conclude that $NE(X)=S_1$, that is a regular cone.

Now, we are going to show that Statement (b) implies Statement (c). If the cone $NE(X)$ is regular, it is clear that it must be generated by the classes of $\tilde{L}$ and the strict transforms of the exceptional divisors. Then $D_m$ is nef because $D_m\cdot \tilde{L}=0$, $D_m\cdot E_i=0$ for $i=1, 2, \ldots,m-1$ and $D_m\cdot E_m=1$.

Next, let us show that statements (c) and (d) are equivalent. Firstly notice that $\nu$ takes negative values for all powers of $x$ with positive exponent. If $f(x,y)\in k[x,y]\setminus k$ and $x$ does not divide $f$, one has that
\begin{equation}\label{www}
f(x,y)=f(1/v,u/v)=\frac{h_f(u,v)}{v^{\deg{(h_f)}}},
\end{equation}
for some polynomial $h_f(u,v)\in k[u,v]$ which corresponds with $f$.  Consider $f\in k[x,y]$ an denote by  $\Phi(f)$ the projective plane curve of degree $\deg{(h_f)}$ defined by the homogeneous polynomial $X^{\deg{(h_f)}}h_f(Y/X,Z/X)$. Set $W$ the projective line of $\gp^2$ with equation $X=0$, then
$\Phi$ defines a one-to-one correspondence between the set $\Gamma$ of non-constant polynomials in $k[x,y]$ (modulo multiplication by a nonzero element of $k$) such that $x$ does not divide them and the set of projective curves of $\mathbb{P}^2$ which have neither the line $L$ nor the line $W$ as components. Then, using Equality (\ref{www}) and Noether formula, we get that for all $f\in \Gamma$,
$$-\nu(f)=D_m\cdot \tilde{\Phi}(f),$$
$\tilde{\Phi}(f)$ being the strict transform of the curve $\Phi(f)$ on $X$. Now, taking also into account that $$D_m\cdot \tilde{L}=0, \;\;\; D_m \cdot\tilde{W}=d_m>0$$ and $D_m\cdot E_i\geq 0$ for all $i\in \{1, 2,\ldots,m\}$, it happens that
$\nu(f)\leq 0$ for all $f\in k[x,y]$ if and only if the divisor $D_m$ is nef.

The above reasoning finishes the proof because the fact that Statement (c) implies (a) is straightforward.

\end{proof}

Lemma \ref{lema2} and the equivalence between statements (a) and (d) of Theorem \ref{gordo1} prove the following result.

\begin{cor}\label{grrr}
With the same notations as in Theorem \ref{gordo1}, assume that $\nu(f)\leq 0$ for all $f\in k[x,y]$ and set $\nu_i$ the divisorial valuation defined by the divisor $E_i$, $i\in \{1, 2, \ldots,m-1\}$. Then $\nu_i(f)\leq 0$ for all $f\in k[x,y]$.
\end{cor}

\begin{rem} \label{llados}
{\rm

Statement (b) of Theorem \ref{gordo1} is equivalent to the fact that the semigroup of effective classes in ${\rm Pic}(X)$ is spanned by the classes of $\tilde{L},E_1, E_2, \ldots, E_m$.
}
\end{rem}

\begin{rem}\label{remark3}
{\rm
Taking into account the above corollary, it holds that any of the statements in Theorem \ref{gordo1} implies that the divisor $D_j$ is effective for all $j\in \{1,2,\ldots,m\}$. Indeed, one can write $D_j$ in the form
$$D_j=(D_j\cdot D_0)\tilde{L}+\sum_{i=1}^m (D_j\cdot D_i)E_i,$$
and $D_j\cdot D_i\geq 0$ for all $0 \leq i \leq m$  because $D_j$ is nef and $D_i$ belongs to $NE(X)$ for all $i\in \{0,1,\ldots,m\}$ (it suffices to consider items (b) and (c) of Theorem \ref{gordo1} and the inclusion $Q(X)\subseteq \overline{NE}(X)$).

Moreover, and as a consequence of what we have said, it holds that the Iitaka dimension, $\kappa(D_j)$, of the divisor $D_j$ satisfies $\kappa(D_j)\in \{0,1,2\}$. Furthermore, $\kappa(D_j)=2$ if and only if $D_j$ is big, that is, if and only if $D_j^2>0$.
}
\end{rem}

\begin{rem}
{\rm The equivalence between items (a) and (c) of Theorem 1 has been proved in \cite[Proposition 9.10]{j}. Assume now that $k = \mathbb{C}$ the field of complex numbers. A characterization of the fact $\nu(f) \leq 0$ for all $f \in \mathbb{C}[x,y]$ is given in \cite{mon} in terms of what the author calls key forms of $\nu$. These are sequences $\{g_j\}_{j=0}^{n+1}$ of elements in $\mathbb{C}[x,x^{-1},y]$ that allows us to compute $- \nu(g)$, $g \in \mathbb{C}[x,x^{-1},y]$, as the minimum of certain weighted degrees
of polynomials in $n+2$ variables  $G$ such that $G(g_0,g_1, \ldots, g_{n+1})= g$. Fixed $\{x,y\}$, $-\nu$ admits a unique sequence of key forms. The mentioned characterization states that $\nu(f) \leq 0$ for all $f \in \mathbb{C}[x,y]$ if, and only if, $-\nu (g_{n+1}) \geq 0$ \cite[Theorem 1.7, item 1]{mon}. Next result shows that for the field $k=\mathbb{C}$ that characterization is equivalent to the one we state in item (a) of our Theorem \ref{gordo1}.
}
\end{rem}

\begin{cor}
\label{mondal1}
Set $k=\mathbb{C}$ and $\nu$ a plane divisorial valuation as above. With the same notation as in Theorem \ref{gordo1} and in the previous remark: $-\nu (g_{n+1}) \geq 0$ if and only if $d_m^2\geq \betabarra_{g+1}$.
\end{cor}

\begin{proof}
We start by noting that, if we consider the system of generators $\{v,u\}$ of the maximal ideal of $R$ (see Section \ref{pdv}) and compute  sequence of key polynomials $\{h_j(v,u)\}_{j=0}^{n+1}$ corresponding to $\nu$ and the mentioned system (see, for instance,  \cite{fj} for the definition), it holds that $$g_{n+1} (x,y)= x^{\deg_u(h_{n+1})} h_{n+1}(1/x,y/x)$$
\cite[Remark 2.6]{mon}. Now, bearing in mind that $h_{n+1}$ is a general element element of the valuation $\nu$ and \cite[Remark 2.4]{fj}, it holds that $-\nu (g_{n+1}) = d_m^2- \betabarra_{g+1}$, which concludes the proof.
\end{proof}

 Recall that, for a curve $C$ on $\gp^2$, $\tilde{C}$ means strict transform on $X$ and, for a divisor $D$ on $X$, $\kappa(D)$ denotes its Iitaka dimension.

\begin{lem}\label{lema3}

Assume that the characteristic of the field $k$ is zero, $D_m^2=0$ and $\kappa(D_m)>0$. Then, there exists a homogeneous polynomial $F\in k[X,Y,Z]$ of degree $d_m$ such that the equation $F(X,Y,Z)=0$ defines a curve on $\gp^2$ having only one place at infinity, the complete linear system $|D_m|$ is base-point-free and the morphism $\varphi_{|D_m|}:X\rightarrow \gp^1$ that it determines eliminates the indeterminacies of the rational map $\gp^2\cdots \rightarrow \gp^1$ induced by the pencil $\langle F, Z^{d_m}\rangle\subseteq H^0(\gp^2,{\mathcal O}_{\gp^2}(d_m))$.

\end{lem}

\begin{proof}
 We assert that the linear system $|\ell D_m|$ has no fixed component for some $\ell\in \mathbb{Z}_{>0}$. Indeed, since $\kappa(D_m)>0$ it holds that $\dim |nD_m|>0$ for some positive integer $n$. If $|nD_m|$ has a fixed part $B\not=0$ then, by Remark \ref{remark3} and the equality $D_m^2=0$, the integral components of $B$ are among $\{\tilde{L}, E_1, E_2, \ldots,E_{m-1}\}$ and, therefore, $D_m\cdot B=0$. Then $D_m\cdot (nD_m-B)=0$ and, hence, $nD_m-B$ is linearly equivalent to $\ell D_m$ for some $\ell\in \mathbb{Z}_{>0}$. The assertion follows because $|nD_m-B|$ has no fixed component.

The above reasoning shows that $|\alpha \ell D_m|$ has no fixed part for all $\alpha\in \mathbb{Z}_{>0}$ and this implies that there exists a curve $C$ on $\gp^2$ such that its strict transform $\tilde{C}$ is linearly equivalent to $\beta D_m$ for some $\beta \in \mathbb{Z}_{>0}$. According with Definition 8 in \cite{cpr2}, $\tilde{C}$ is a {\it numerical $m$-curvette}  and, since $\tilde{C}^2=0$, Theorem 1 in the same paper proves that $C$ has only one place at infinity and $\tilde{C}$ is linearly equivalent to $D_m$.

Finally, it holds that the degree of $C$ is $d_m$ and  $d_m^2=\sum_{i=1}^m \mathrm{mult}_{p_i}(C)^2$, where $\mathrm{mult}_{p_i}(C)$ denotes the multiplicity at $p_i$ of the strict transform of $C$ at that point, because $D_m^2=0$. This concludes our result by \cite[Section 4]{cpr}.

\end{proof}

Theorem 1 characterizes the case when $\nu(f) \leq 0$ for all polynomials in $k[x,y]\setminus k$, $k$ being an algebraically closed field of arbitrary characteristic. Next, we provide two characterizations for the case when the inequality is strict.

\begin{teo}\label{gordo2}
Keep the same assumptions and notations as in Theorem \ref{gordo1}. Then, the following conditions are equivalent:
\begin{itemize}
\item[(a)] Either $d_m^2>\betabarra_{g+1}$, or $d_m^2=\betabarra_{g+1}$ and $\kappa(D_m)=0$.
\item[(b)] $D_m\cdot \tilde{C}>0$ for any integral curve $C$ on $\mathbb{P}^2$ different from $L$, $\tilde{C}$ being the strict transform on $X$ of the curve $C$.
\item[(c)] $\nu(f)< 0$ for all $f\in k[x,y]\setminus k$.

\end{itemize}
Moreover, when the characteristic of $k$ is zero, Condition {\rm (a)} can be replaced by
\begin{itemize}
\item[(a')] Either $d_m^2>\betabarra_{g+1}$, or $d_m^2=\betabarra_{g+1}$ and $\dim|D_m|=0$.

\end{itemize}

\end{teo}

\begin{proof}
Assume that item (a) holds and let us prove item (b). Reasoning by contradiction, let $C$ be an integral curve on $\mathbb{P}^2$ different from $L$ and such that $D_m\cdot \tilde{C}\leq0$. By Theorem \ref{gordo1}, $D_m$ is a nef divisor and so $D_m \cdot \tilde{C} = 0$. This implies that the class $[\tilde{C}]$ belongs to the face of the cone of curves $NE(X)$ given by
$[D_m]^{\perp}\cap NE(X)$, where $[D_m]^{\perp}$ means the orthogonal space to the vector $[D_m]$ (with respect to the intersection product). With the previous notations, let us notice that this face, $F$, is spanned by $[\tilde{L}], [E_1], [E_2], \ldots, [E_{m-1}]$. Since these last classes are the unique on $F$ given by curves with negative self-intersection (by the proof of Theorem \ref{gordo1}), one has that $\tilde{C}^2=0$. This implies that $D_m^2=0$ and $\tilde{C}$ is linearly equivalent to a multiple of $D_m$. Therefore $D_m^2=0$ and $\kappa(D_m)\not=0$ (taking into account Remark \ref{remark3}), which is a contradiction because our Statement (a) is equivalent to say that either $D_m^2>0$, or $D_m^2=0$ and $\kappa(D_m)=0$.

A similar reasoning as that of the proof of (c) implies (d) in Theorem \ref{gordo1} allows us to show that item (b) in this theorem implies item (c).

Next we prove, also by contradiction, that item (c) implies item (a). We will use the notation of the proof of Theorem \ref{gordo1}. By that theorem,  we have that $d_m^2\geq \betabarra_{g+1}$. If $D_m^2=0$ (i.e., $d_m^2 =\betabarra_{g+1}$) and $\kappa(D_m)>0$, then there exists a positive integer $n$ such that $nD_m-\tilde{R}$ gives an effective class,  $R$ being an irreducible curve of $\mathbb{P}^2$ different from $L$ and $W$. Therefore $D_m\cdot \tilde{R}\leq 0$ and this implies that $\nu(f)\geq 0$, where $f\in \Gamma$ and $\tilde{R}=\tilde{\Phi}(f)$, which gives the desired contradiction.

We conclude by proving that, assuming that the characteristic of $k$ is zero, items (a) and (a') are equivalent. Indeed, the fact that (a) implies (a') is clear and the converse implication holds by Lemma \ref{lema3}.

\end{proof}

\begin{rem}
{\rm
An equivalent condition to the fact that $\nu(f) <0$ for all polynomials $f \in \mathbb{C}[x,y] \setminus \mathbb{C}$, given in terms of the above mentioned sequence of key forms $\{g_j\}_{j=0}^{n+1}$ of $\nu$, has been given in \cite[Theorem 1.4, item 2]{mon}. The condition is that either $-\nu(g_{n+1}) > 0$ or $\nu(g_{n+1}) = 0$ and $g_{n+1} \not \in \mathbb{C}[x,y]$. The following result proves that, when $k = \mathbb{C}$, the above condition and item (a) in Theorem \ref{gordo2} are equivalent.
}
\end{rem}

\begin{cor}
\label{mondal2}
Let $\nu$ be a divisorial valuation as in Theorem \ref{gordo2}. Assume that $k$ is the field of complex numbers. Consider the above defined divisor $D_m$ on $X$ and the sequence $\{g_j\}_{j=0}^{n+1}$ of key forms of $\nu$. Then the following conditions are equivalent:

\begin{itemize}
\item[(a)]  $\nu(g_{n+1}) = 0$ and $g_{n+1} \not \in \mathbb{C}[x,y]$.
\item[(b)] $D_m^2=0$ and $\kappa(D_m)=0$.
\end{itemize}
\end{cor}
\begin{proof}
By \cite[Remark 2.6]{mon}, we have that $$g_{n+1} (x,y)= x^{\deg_u(h_{n+1})} h_{n+1}(1/x,y/x),$$
where $h_{n+1}(v,u)$ is the last key polynomial attached with the valuation $\nu$ (it is a general element element of $\nu$). It is clear from this equality that $g_{n+1} (x,y)\not\in \mathbb{C}[x,y]$ if and only if $\deg(h_{n+1})\not=\deg_u(h_{n+1})$. This happens if and only if $\deg(h_{n+1})>d_m$. Indeed, it suffices to apply Bézout's Theorem to the line at infinity $L$ and the projective closure $H$ of the affine curve defined by $h_{n+1}(v,u)=0$ and take into account  that, by \cite[Remark 2.4]{fj}, $d_m=\deg_u(h_{n+1})$.

Now,  Lemma \ref{lema3} and \cite[Proposition 3.8]{mon2} show that the inequality $\kappa(D_m)>0$ implies that $g_{n+1} \in \mathbb{C}[x,y]$. Therefore (b) can be deduced from (a) since $\nu(g_{n+1}) = 0$ is an equivalent fact to $D_m^2=0$.

The converse implication follows from the fact that if $g_{n+1}\in \mathbb{C}[x,y]$, then $\deg(h_{n+1})=d_m$. Thus, the strict transform on $X$ of the above curve $H$ is linearly equivalent to $D_m$ and this implies that $\kappa(D_m)>0$.

\end{proof}

\section{Characteristic cone and Cox ring for non-positive valuations}
\label{latres}
In this section, we consider  rational surfaces $X$ given by plane divisorial valuations of the quotient field of $R$ with good behaviour al infinity, that is satisfying the equivalent statements given in Theorem \ref{gordo1}. We will characterize when the characteristic cone $\tilde{P}(X)$ of these surfaces is closed and, as a consequence, its Cox ring, Cox$(X)$, is a finitely generated $k$-algebra.

\begin{teo}\label{gordo3}

Let $X$ be a surface defined by a plane divisorial valuation  $\nu$ as in Theorem \ref{gordo1}. Assume that the equivalent statements given in that theorem  happen. Consider the sequence of divisors on $X$ given in (\ref{Z}).  Then, the characteristic cone $\tilde{P}(X)$ is closed if and only if, for all $i\in \{2,\ldots,m\}$, either $D_i^2>0$, or $D_i^2=0$ and $\kappa(D_i)> 0$. If the characteristic of the field $k$ is zero, the condition $\kappa(D_i)> 0$ can be replaced by $\dim |D_i|>0$.
\end{teo}

\begin{proof}

For a start, let us see that any of the facts $D_i^2>0$, or $D_i^2=0$ and $\kappa(D_i)\neq 0$ for all index $i\geq 2$ imply that $\tilde{P}(X)$ is closed. To do it, we will prove that, for all $i\in \{1,2,\ldots,m\}$, the complete linear system $|nD_i|$ is base-point-free for $n$ large enough. We use complete induction on $i$.
For $i=1$ our assertion is clear,  so suppose that it is true for $1, 2, \ldots,i-1$ and let us show it for $i$.

Along this proof we will use the concept of contractible curve $C$ on $X$, which means that $C \subset X$ can be contracted to another algebraic $k$-scheme. We will also use the concept of negative definite curve  on $X$, that is,  a curve whose intersection matrix for its irreducible components is negative definite.

Without loss of generality, we can assume that $i=m$.

{\it Firstly, suppose $D_m^2 > 0$}. Consider the reduced curve
$$F:=\tilde{L}+\sum_{i=1}^{m-1} E_i.$$
Assume that $\tilde{L}$ meets the divisor $E_r$. If $m= r$ then, applying Artin's contractibility criterion, $F$ is contractible because it is a Hirzebruch-Jung string (see \cite[Sections III.2 and III.3]{barth}, for instance). Let $f:X\rightarrow Y$ be such a contraction. Using Stein factorization \cite[III, Corollary 11.5]{hart} we can assume that $f_*\co_X=\co_Y$.

Now, consider, on the one hand, the set of Stein factors of the structural morphism $h:X\rightarrow {\rm Spec}(k)$ of $X$ as $k$-scheme ordered by the domination relation and, on the other hand, the set of topological cells of the characteristic cone $\tilde{P}(X)$ ordered by inclusion. The first set is that of factorizations $h=q\circ p$ where $p:X\rightarrow Y$ is surjective, $q:Y \rightarrow {\rm Spec}(k)$ is projective and $p_*\co_X=\co_Y$.  By a result of Kleiman \cite[page 340]{Kleiman}, there exists an order isomorphism between the above mentioned sets.

Let $g:Y\rightarrow {\rm Spec}(k)$ be the structural morphism of $Y$ as $k$-scheme; clearly $h=g\circ f$ is a Stein factor of $h$. We claim that its image by the mentioned isomorphism is the cell of $\tilde{P}(X)$ defined by the ray $\mathbb{Q}_+[D_m]$. Indeed, the multiples of $D_m$ are the unique divisors $D$ on $X$ such that $D\cdot M=0$ for every irreducible component $M$ of $F$; so any contraction of $F$ must be defined by global sections of a sheaf ${\mathcal O}_{X}(\alpha D_m)$ for some positive integer $\alpha$. Therefore $[D_m]\in \tilde{P}(X)$.

Hence we suppose from now on that $m>r$ and divide this part of the proof in two cases depending on whether $p_m$ is, or not, a satellite point.

{\it Assume that $p_m$ is a satellite point}. In this case $F$ has two connected components and there are two exceptional divisors $E_a$ and $E_b$ meeting $E_m$ which belong to different connected components of $F$. By induction hypothesis we know that, for each $i\in \{0,1, \ldots,m-1\}$, there exists a positive integer $n_i$ such that $|n_iD_i|$ is base-point-free and, therefore, there exists a curve $C_i\not= L$ on $\mathbb{P}^2$ such that its strict transform $\tilde{C}_i$ is linearly equivalent to $n_iD_i$ (where $n_0:=1$ and $C_0=E_0$). We can assume, moreover, that $n_a>-E_a^2$ and $n_b>-E_b^2$.

Let $F'$ be a connected component of $F$ and we are going to prove that $F'$ is contractible. To do it, we will use \cite[Theorem 3.3]{Schroer}, which asserts that it suffices  to check that $F'$ is definite negative and there exists a curve $\Delta \subset X$ disjoint to $F'$ with $\Delta \cdot R >0$ for every curve $ R \subset X$ not supported by $F'$.

Without loss of generality, suppose that $E_b$ is not supported by $F'$ and consider the effective divisor
$$\Delta:=E_b+\sum_{i\in {\mathcal J}(F')} \tilde{C}_i,$$
where ${\mathcal J}(F')$ is the set of indices $1 \leq i \leq m-1$ corresponding with divisors $E_i$ which are not in $F'$, except when $\tilde{L}$ is not a component of $F'$; in this last case, to get ${\mathcal J}(F')$, one must add the index $i=0$ to the above ones. Clearly, $b\in {\mathcal J}(F')$ and $\Delta$ is disjoint to $F'$.

Let us show that $\Delta\cdot R>0$ for every curve $R\subseteq X$ not supported by $F'$. To see this, it suffices to consider three types of curves $R$.

For a start, let $R$ be the strict transform of a curve on $\mathbb{P}^2$ such that $\tilde{L}$ is not supported by it. It is clear that $E_b\cdot R\geq 0$; moreover, $D_i\cdot R\geq 0$ for all $1 \leq i \leq m-1$ by Corollary \ref{grrr} and Theorem \ref{gordo1}. If there exists some $i\in {\mathcal J}(F')$ such that $D_i^2>0$ then, by Theorem \ref{gordo2}, it holds that $D_i\cdot R>0$ and we conclude the desired inequality: $\Delta\cdot R>0$. Otherwise, Lemma \ref{lema2} implies that
${\mathcal J}(F')=\{b\}$ and $D_b^2=0$. In addition, this last equality and the nefness of the divisor $D_b$ show that $D_{b}\cdot D_i>0$ for all $i\in \{1, 2, \ldots,m\}\setminus \{b\}$; notice that this happens because $D_i^2\geq 0$ and the class $[D_i]$ is not a multiple of $[D_{b}]$. Also, Remark \ref{remark3} proves that
$$D_{b}=(D_{b}\cdot D_0)\tilde{L}+\sum (D_{b}\cdot D_i)E_i,$$ where the sum is taken over the set  $i\in \{1, 2, \ldots,m\}\setminus \{b\}$. Then either $D_b\cdot R>0$,  or $\tilde{L}\cdot R=0$ and $E_i\cdot R=0$ for all $i\in \{1, 2, \ldots,m\}\setminus \{b\}$ (what implies that $E_{b}\cdot R>0$). Hence we also conclude that $\Delta\cdot R>0$.


Now we consider our second type of curves $R$, which are the exceptional ones. First, let us show that $\Delta\cdot E_i>0$ for  $i\in {\mathcal J}(F')$. In fact, if  $\delta_{ij}$ denotes the Kronecker's delta, it holds $D_j\cdot E_i=\delta_{ij}$ for all $j\in {\mathcal J}(F')$, $E_b\cdot E_i\geq 0$ if $i\not=b$ and $(n_bD_b+E_b)\cdot E_b>0$. Finally, $\Delta\cdot E_m>0$ because $E_b\cdot E_m=1$ and $D_i\cdot E_m=0$ for all $i\in {\mathcal J}(F')$.

To finish we have to consider curves $R$ such that $\tilde{L}$ is supported by them and not supported by $F'$. Then it suffices to note that $\Delta\cdot \tilde{L}\geq E_0\cdot \tilde{L}>0$, which concludes the reasoning for our third type of curves and this part of the proof.

To finish of checking the  hypotheses of Theorem 3.3 in \cite{Schroer}, it remains to show that the curve $F'$ is negative definite. This is true because the class in
${\rm Pic}_{\mathbb{Q}}(X)$ of any effective divisor whose components are integral components of $F'$ belongs to the face of the cone of curves $NE(X)$ given by $NE(X)\cap [D_m]^{\perp}$, which is contained in the complementary of the set $Q(X)$ defined before Lemma \ref{lema1} (since $D_m^2>0$).

As a consequence of the above paragraphs, we conclude that $F$ is contractible because every connected component is \cite[Lemma 3.1]{Schroer}. Finally, reasoning as above, a contraction of $F$ must be defined by sections of $\co_X(\alpha D_m)$ for some positive integer $\alpha$ and, therefore, $[D_m]\in \tilde{P}(X)$.


{\it Suppose now that $p_m$ is a free point}. We will prove, reasoning by contradiction, that $|nD_m|$ is base point free for some positive integer $n$, which will conclude the proof in the case $D_m^2 > 0$. By \cite[Theorem 6.1]{Zariski}, it is enough to assume that $|nD_m|$ has a fixed part for all $n$ and lead to a contradiction. Assuming this and $n$ large enough, it holds that only $\tilde{L}, E_1, E_2, \ldots, E_{m-1}$ can be fixed components of $|nD_m|$. In fact, it suffices to apply \cite[Theorem 9.1]{Zariski} and take into account Remark \ref{remark3} and the equalities $D_m\cdot E_i=0$ for $i\in \{1, 2, \ldots,m-1\}$, $D_m\cdot E_m=1$ and $D_m\cdot \tilde{L}=0$.

Now, pick $n \gg 0$ such that the last condition holds and the linear system $|nD_{m-1}|$ is base-point-free. We can do this by using the induction hypotheses. Since $p_m$ is a free point, the local equations at it of the curves in ${\pi_m}_*|nD_m|$ coincide with those of the curves in ${\pi_m}_*|nD_{m-1}|$. Also ${\pi_m}_*|nD_{m-1}|$ is base-point-free because of the natural isomorphism with $|nD_{m-1}|$. The above facts and the shape of the local equations at $p_m$ of the curves in ${\pi_m}_*|nD_{m-1}|$ allows us  to deduce that $E_{m-1}$ is not a component of some curve in the complete linear system  $|nD_m|$. Therefore the integral fixed components of $|nD_m|$ must belong to the set $S:=\{\tilde{L}, E_1, E_2, \dots,E_{m-2}\}$. If $R$ is one of them and $R'$ is some other curve in $S\cup \{E_{m-1}\}$ that meets $R$, then $(nD_m-R)\cdot R'=-R\cdot R'=-1$. This  and the fact that the above introduced curve $F$ is connected show that $F$ is in the fixed part of $|nD_m|$ and, in particular, $E_{m-1}$ is a fixed component, which provides the desired contradiction.




To complete the proof of this implication, we have to {\it assume that  $D_m^2=0$ and $\kappa(D_m)> 0$}. These conditions imply that, for some positive integer $n$, $nD_m$ is effective and linearly equivalent to a divisor $H+T$, where $T$ is the fixed part of $|nD_m|$ and $H$ an effective divisor such that the complete linear system $|H|$ has no fixed part. Since $0=n^2D_m^2=nD_m\cdot H+nD_m\cdot T$ and $D_m$ is nef, it holds that $D_m\cdot H=0$ and, hence, the class $[H]$ is a multiple of $[D_m]$. Therefore, again by \cite[Theorem 6.1]{Zariski}, $|nD_m|$ has no base point for $n$ large enough.

The {\it converse of our first statement} follows straightforward from the fact that $\tilde{P}(X)$ is not closed if $\kappa(D_i)=0$ for some $i\in \{1,2,\ldots,m\}$. 

Finally, it is also straightforward that our last statement holds by Lemma \ref{lema3} which finishes the proof.

\end{proof}

As a consequence of the above theorem, \cite[Corollary 3]{Galmon} and \cite{Hu}, we get the following result that gives conditions for the ring Cox$(X)$ of our rational surfaces $X$ is a finitely generated $k$-algebra.

\begin{cor}
\label{COX}
 Keep the assumptions and notations of Theorem \ref{gordo3}. Then the Cox ring, $\mathrm{Cox}(X)$, is a finitely generated $k$-algebra if and only if, for all $i\in \{2,\ldots,m\}$, either $D_i^2>0$, or $D_i^2=0$ and $\kappa(D_i)> 0$. If the characteristic of the field $k$ is zero, the condition $\kappa(D_i)> 0$ can be replaced by $\dim |D_i|>0$.

\end{cor}

The above result provides a wide range of rational surfaces with finitely generated Cox ring. Moreover their anticanonical Iitaka dimension can be $-\infty$, as the next example shows. It gives a new infinite family of surfaces with both conditions and with arbitrarily large Picard number.

\begin{exa}
\label{ejemplo}
{\rm
Assume that the characteristic of the field $k$ is zero. Fix two positive integers $a$ and $r$ such that $r\geq a\geq 4$ and $\gcd(a,r+1)=1$. Let $X$ be a surface obtained by a sequence of blow ups as in (\ref{blow}) coming from a divisorial valuation $\nu$ with 3 maximal contact values, $\betabarra_0:=a$, $\betabarra_1:=ar^2-r-1$ and $\betabarra_2:=\betabarra_0\betabarra_1+1$, and such that the strict transforms of the line at infinity $L$ pass exactly through the first $r$ blown-up points. Notice that $\gcd(\betabarra_0,\betabarra_1)=1$ and $r\leq \lfloor \betabarra_1/\betabarra_0\rfloor$; therefore such a valuation exists.

Let $p_m$ be the last blown-up point, that is, $\nu$ is the divisorial valuation defined by the exceptional divisor $E_m$. Then $d_m=ra$ and, therefore, $d_m^2>\betabarra_2$, that is, $D_m^2>0$. By Lemma \ref{lema2},  $D_i^2\geq 0$, $2 \leq i \leq m$, and, moreover, the equality $D_i^2=0$ only can occur when $i=i_1:=\lfloor \betabarra_1/\betabarra_0\rfloor +1$ (that is, the index $i_1$ is the only one where $p_{i_1}$ is a free point and $P_{i_1+1}$ a satellite one). Now,
$D_{i_1}^2=r^2-i_1>0$, hence $\nu$ satisfies the hypotheses of Corollary \ref{COX} and, as a consequence, the Cox ring, Cox$(X)$, is finitely generated.

If $K_X$ denotes a canonical divisor of $X$ we have that $D_{i_1}\cdot K_X=-3r+i_1>0$. Since $D_{i_1}$ is nef we deduce that $[-K_X]\not\in NE(X)$ and, as a consequence, the anticanonical Iitaka dimension of $X$ is $-\infty$.

Notice that $X$ cannot be obtained with the procedure described in \cite{cpr,cpr2} to get surfaces with finitely generated Cox ring. Indeed, if this happened, by the Abhyankar-Moh Theorem \cite{ja4, pinkam}, there would exist a curve $C$ on $\mathbb{P}^2$ with only one place at infinity such that the morphism $\pi$ given by $\nu$ is a resolution of its singularity at infinity and whose semigroup at infinity $- \nu (k[x,y])$
is spanned by $\delta_0:=-\nu(x)=r\betabarra_0$, $\delta_1:=-\nu(y)=r\betabarra_0-\betabarra_0=(r-1)\betabarra_0$ and $\delta_2:=-\nu(G(x,y,1))$, where $G(X,Y,Z)$ is a homogeneous polynomial of degree $r$ defining a curve whose local equation at $p$ defines an analytically irreducible germ whose strict transform on $X$ meets $E_{i_1}$ transversally at a non-singular point of the exceptional locus of $\pi$. Moreover, the  condition that $\gcd(\delta_0,\delta_1)\delta_2=\betabarra_0\delta_2$ belongs to the semigroup spanned by $\delta_0$ and $\delta_1$ must be satisfied. Using (\ref{www}) we get
$$\delta_2=r\nu(v)-\betabarra_1=r^2\betabarra_0-\betabarra_1=r+1.$$
As a consequence,  $r+1$ must
belong to the semigroup spanned by $r$ and $r-1$, which is a contradiction.


}
\end{exa}


\begin{thebibliography}{99}
\footnotesize \setlength{\baselineskip}{3mm}

\bibitem{ja1} S.S. Abhyankar, Local uniformization on algebraic surfaces over ground field of characteristic $p \neq 0$, {\it Ann. Math.} {\bf 63} (1956) 491--526.

\bibitem{ja2} S.S. Abhyankar, On the valuations centered in a local domain, {\it Amer. J. Math.} {\bf 78} (1956) 321--348.

\bibitem{ja4} S.S. Abhyankar, T.T. Moh, Newton-Puiseux expansion and generalized Tschirnhausen transformation I and II, {\it J. reine angew. Math.} {\bf 260} (1973) 47--83 and {\bf 261} (1973) 29--54.


\bibitem{ja6} S.S. Abhyankar, T.T. Moh, On the semigroup of a meromorphic curve, Proc. Int. Symp. Algebraic Geom. Kyoto (1977) 249--414.

\bibitem{al} M. Artebani, A. Laface, Cox rings of surfaces and anticanonical Iitaka dimension, {\it Adv. Math.} {\bf 226} (2011) 5252-5267.

\bibitem{barth} W. Barth, C. Peters, A. Van de Ven, {\it Compact complex surfaces}, Ergebnisse der Math.  {\bf 4}, Springer-Verlag, 1984.

\bibitem{bchm} C. Birkar, P. Cascini, C.D. Hacon, J. McKernan, Existence of minimal models for varieties of log general type, {\it J. Amer. Math Soc.} {\bf 23} (2010) 405--468.


\bibitem{c-g} A. Campillo, G. González-Sprinberg, On characteristic cones, clusters and chains of infinitely near points. {\it Progr. Math.} {\bf 162} (1998) 251--261.


\bibitem{cpr} A. Campillo, O. Piltant, A. Reguera, Curves and divisors on surfaces associated to plane curves with one place at infinity, {\it Proc. London Math. Soc.} {\bf 84} (2002) 559--580.

\bibitem{cpr2} A. Campillo, O. Piltant, A. Reguera, Cones of curves and of line bundles ``at infinity'', {\it J. Algebra} {\bf 293} (2005) 513--542.

\bibitem{tot10} A.M. Castravet, J. Tevelev, Hilbert's 14th problem and Cox rings. {\it Compos. Math.} {\bf 142} (2006) 1479--1498.



\bibitem{cox5} D.A. Cox, The homogeneous coordinate ring of a toric variety, {\it J. Algebraic Geom.} {\bf 4} (1995) 17--50.

\bibitem{he3} S.D. Cutkosky, Complete ideals in algebra and geometry, {\it Contemp. Math.} {\bf 159} (1994) 27--39.

\bibitem{d-g-n} F. Delgado, C. Galindo, A. Nu\~nez, Saturation for valuations
on two-dimensional regular local rings, {\em Math. Z.} {\bf 234}
(2000) 519--550.

\bibitem{fa} L. Facchini, V. Gonz\'alez-Alonso, M. Las\'on, Cox rings of du Val singularities, {\it Le Matematiche} {\bf 66} (2011) 115--136.

\bibitem{fj} C. Favre, M. Jonsson, {\it The valuative tree,} Lect. Notes Math. {\bf 1853}. Springer-Verlag, 2004.

\bibitem{fj2} C. Favre, M. Jonsson, Eigenvaluations, {\it J. Ann. Sci. \'Ecole Norm. Sup.} {\bf 40} (2007) 309--349.

\bibitem{fj3} C. Favre, M. Jonsson, Dynamical compactifications of $\mathbb{C}^2$, {\it Ann. Math.} {\bf 173} (2011), no. 1, 211--248.

\bibitem{ijm} C. Galindo, F. Monserrat, On the cones of curves and of line bundles of a rational surface,
{\it Int. J. Math.} {\bf 15} (2004) 394--407.

\bibitem{gm} C. Galindo, F. Monserrat, The total coordinate ring of a smooth projective surface, {\it J. Algebra} {\bf 284} (2005) 91--101.

\bibitem{hel} C. Galindo, F. Monserrat, The cone of curves associated to a plane configuration, {\it Comm. Math. Helv.} {\bf 80} (2005) 75--93.

\bibitem{london} C. Galindo, F. Monserrat, $\delta$-sequences and evaluation codes defined by plane valuations at infinity, {\it Proc. London Math. Soc.} {\bf 98} 	(2009) 714--740.

\bibitem{Galmon} C. Galindo, F. Monserrat, The Abhyankar-Moh theorem for a plane valuation at infinity, {\it J. Algebra} {\bf 374} 	(2013) 181--194.

\bibitem{dcc} C. Galindo, F. Monserrat, Evaluation codes defined by finite families of plane valuations at infinity, {\it Des. Codes, Crypt.}  {\bf 70} (2014) 189--213.

 \bibitem{ja15} S. Greco, K. Kiyek, General elements in complete ideals and valuations centered at a two-dimensional regular local ring, in: Algebra, Arithmetic, and Geometry, with Applications, Springer, 2003, 381--455.

\bibitem{hart} R. Hartshorne, {\it Algebraic Geometry},
Springer-Verlag, 1977.


 \bibitem{Hu} Y. Hu, S. Keel, Mori dream spaces and GIT, {\it Michigan Math. J.} {\bf 48} (2000), 331--348.

\bibitem{hw} D. Hwang, J. Park, Redundant blow-ups and Cox rings of rational surfaces, Preprint arXiv:1303.2274.
    
\bibitem{j}  M. Jonsson, Dynamics on Berkovich spaces in low dimensions, Preprint arXiv:1201.1944. 

 \bibitem{he5} Y. Kawamata, The cone of curves of algebraic varieties, {\it Ann. Math.} {\bf 119} (1984) 603--633.

 \bibitem{to19} Y. Kawamata, On the cone of divisors of Calabi-Yau fiber spaces, {\it Int. J. Math.}
{\bf 8} (1997) 665--687.

\bibitem{Kleiman} S. Kleiman, Towards a numerical theory of ampleness, {\it Math. Ann.} {\bf 84}
(1966) 293--344

 \bibitem{l-m} M. Lahyane, B. Harbourne, Irreducibility of $(-1)$-classes on anticanonical rational surfaces and finite generation of the effective monoid, {\it Pacific J. Math.} {\bf 218} (2005) 1--14.

\bibitem{mon} P. Mondal, How to determine the sign of a valuation on $\mathbb{C}[x,y]$? Preprint arXiv:1301.3172.

\bibitem{mon2} P. Mondal, Mori dream surfaces associated with curves with one place at infinity, Preprint arXiv:1312.2168.

 \bibitem{he9} S. Mori, Threefolds whose canonical bundles are not numerically effective, {\it Ann. Math.}
{\bf 116} (1982) 133--176.

\bibitem{to29} D. Morrison, Compactifications of moduli spaces inspired by mirror symmetry, {\it Ast\`{e}risque} {\bf 218} (1993) 243--271.

\bibitem{to30} D. Morrison, Beyond the K\"{a}hler cone, Proc. of Hirzebruch 65 Conf. Algebraic Geom.,  Bar-Ilan Univ. (1996) 361--376.

\bibitem{cox16} S. Mukai, Counterexample to Hilbert's fourteenth problem for the 3-dimensional additive group, RIMS preprint 1343 (2001).

\bibitem{to33} Y. Namikawa, Periods of Enriques surfaces, {\it Math. Ann.} {\bf 270} (1985) 201--222.

 \bibitem{ni} V.V. Nikulin, A remark on algebraic surfaces with polyhedral Mori cone, {\it Nagoya Math.
J.} {\bf 157} (2000) 73--92.

\bibitem{ni-zu} I. Niven, H. Zuckerman, H. Montgomery, {\it An introduction to the
theory of numbers (fifth edition)}, John Wiley \& Sons, 1991.

\bibitem{ot} J.C. Ottem, On the Cox rings of $\gp^2$ blown up in points on a line, {\it Math. Scan.} {\bf 109} (2011) 22--30.

\bibitem{pinkam} H. Pinkham, Courbes planes ayant une seule place à l'infini, Séminaire sur les singularités des surfaces, Preprint, École Polytechnique, Paris, 1977.

\bibitem{Schroer} S. Schr\"{o}er, On contractible curves on normal surfaces, {\it J. reine angew. Math.} \textbf{524} (2000) 1--15.

\bibitem{ja20} M. Spivakovsky, Valuations in function fields of surfaces, {\it Amer. J. Math.} {\bf 112} (1990) 107--156.

\bibitem{to41} H. Sterk, Finiteness results for algebraic K3 surfaces, {\it Math. Z.} {\bf 189} (1985)
507--513.

\bibitem{tei} B. Teissier, Valuations, deformations, and toric geometry. Valuations theory and its applications. Fields Inst. Comm. {\bf 33} (2003) 361--459. Amer. Math. Soc. Providence R.I.

\bibitem{tot41} D. Testa, A. V\'arilly-Alvarado, M. Velasco, Big rational surfaces, {\it Math. Ann.} {\bf 351} (2011) 95--107.

\bibitem{to2} B. Totaro, Hilbert's fourteenth problem over finite fields and a conjecture on the cone of curves, {\it Compos. Math.} {\bf 144} (2008) 1176--1198.

\bibitem{to} B. Totaro, The cone conjecture for Calabi-Yau pairs in dimension 2, {\it Duke Math. J.} {\bf 154} (2010) 241--263.

\bibitem{ja21} O. Zariski, The reduction of singularities of an algebraic surface, {\it Ann. Math.} {\bf 40} (1939) 639--689.

\bibitem{ja22} O. Zariski, Local uniformization on algebraic varieties, {\it Ann. Math.} {\bf 41} (1940) 852--896.
\bibitem{ja23} O. Zariski, P. Samuel, {\it Commutative Algebra}, vol. II, Springer-Verlag, 1960.

\bibitem{Zariski} O. Zariski, The theorem of Riemann-Roch for high multiples of an effective divisor on an algebraic surface, {\it Ann. Math.} \textbf{76} (1962) 560--615.



























\end{thebibliography}
\end{document}